\documentclass[11pt]{article}
\usepackage{fullpage}
\usepackage{diagbox}
\usepackage{mathtools}
\usepackage{bbm}
\usepackage{latexsym}
\usepackage{epsfig}
\usepackage{amsmath,amsthm,amssymb,enumerate}
\usepackage[normalem]{ulem}

\usepackage[a-1b]{pdfx}
\usepackage{hyperref}

\def\es{\emptyset}

\usepackage{color}

\def\cH{{\mathcal H}}

\def\a{\alpha} \def\b{\beta}  
\def\e{\varepsilon} \def\f{\phi}   \def\g{\gamma}
\def\G{\Gamma}

  \def\s{\sigma} 
   \def\Om{\Omega}

\newtheorem{theorem}{Theorem}
\newtheorem{lemma}[theorem]{Lemma}
\newtheorem{corollary}[theorem]{Corollary}

\newtheorem{claim}{Claim}
\newtheorem{definition}{Definition}

\newcommand{\rdup}[1]{{\left\lceil #1\right\rceil }}

\newcommand{\brac}[1]{\left(#1\right)}

\newcommand{\bfrac}[2]{\left(\frac{#1}{#2}\right)}

\newcommand{\set}[1]{\left\{#1\right\}}
\def\sm{\setminus}

\def\es{\emptyset}

\def\Pr{\mathbb{P}}

\newcommand{\ignore}[1]{}

\def\cB{{\mathcal B}}

\def\cD{{\mathcal D}}

\def\cH{{\mathcal H}}

\def\cO{{\mathcal O}}

\newcommand{\beq}[2]{\begin{equation}\label{#1}#2\end{equation}}

\usepackage{tikz}
\usetikzlibrary{decorations.pathmorphing}
\usetikzlibrary{positioning}
\usetikzlibrary{arrows,automata}
\usetikzlibrary{shapes.misc}
\usetikzlibrary{backgrounds}
\usetikzlibrary{arrows,shapes}

\def\dg{{\rm dang}}
\def\bdg{\overline{\dg}}

\def\bc{{\bf c}}
\def\br{{\bf r}}

\begin{document}

\date{}
\title{Some Maker-Breaker games on graphs/hypergraphs}
\author{Patrick Bennett\thanks{Department of Mathematics, Western Michigan University, Kalamazoo MI 49008-5248, Research supported in part by Simons Foundation Grant \#848648.}\and Alan Frieze\thanks{Department of Mathematical Sciences, Carnegie Mellon University, Pittsburgh PA15213,Research supported in part by NSF grant DMS2341774.}\and Wesley Pegden\thanks{Department of Mathematical Sciences, Carnegie Mellon University, Pittsburgh PA15213,Research supported in part by NSF grant DMS1700365.}}

\maketitle
\begin{abstract}
    We consider some biased Maker-Breaker games. Starting with the complete $k$-uniform hypergraph on $n$ vertices, at each turn Maker claims one edge, and then Breaker claims $b$ edges. Maker's goal is to obtain a set of edges having some increasing property. We consider the case where Maker's desired property is minimum degree $m$, or having a full-rank incidence matrix, or having a Berge Hamilton cycle or having a perfect matching or a Hamilton $\ell$-cycle. 
\end{abstract}

\section{Introduction}
In this paper we consider some Maker-Breaker games played on the edges of the complete $k$-uniform hypergraph $H=H_{n,k}$, sometimes referred to as the {\em board}. There are two players, Maker and Breaker, who take turns in acquiring the edges of $H$. In one turn Maker acquires one edge and in one turn Breaker acquires $b$ edges for some fixed number $b$ which we call the {\em bias}. Maker's aim is to acquire a copy of some particular class of hypergraph and Breaker's aim is to prevent this. There will be a value $b^*$, which we call the {\em threshold bias}, such that Maker wins if $b\leq b^*$ and Breaker wins if $b>b^*$. We discuss estimates for $b^*$ for a variety of games. 

The study of biased Maker-Breaker games begins with Chv\'atal and Erd\H{o}s \cite{CE}. There are many results in this area and the book by  Hefetz, Krivelevich, Stojakovi\'c and Szab\'o \cite{Book} provides a very nice introduction to the area.

We first consider the Minimum Degree Game, which is one of the most fundamental of games. Here Maker's goal is to build a hypergraph with minimum degree $m$. (The degree of a vertex $v$ is the number of edges that contain $v$.) For graphs, Gebauer and Szab\'o \cite{GS} asymptotically estimated the threshold bias. We generalize their result to $k$-uniform hypergraphs.  
\begin{theorem}\label{thdegree}
Suppose first that $m=o(\log n)$. Let $b=\frac{\b N}{\log n}$ where $N=\binom{n}{k-1}$. Then for any fixed $\e > 0$, Maker wins the Minimum Degree Game (i.e.~Maker gets minimum degree $m$) if $\b\leq 1-\e$, and Breaker wins if $\b\geq 1+\e$. In general, if $m=\Omega(\log n)$, we only claim that Maker wins if $b\leq \frac{\g N}{km+\log n}$ for $\g<1$.
\end{theorem} 
Next we consider the {\em Matrix Rank Game}. We use the same board $H_{n,k}$ but we view it as an $n\times M,M=\binom{n}{k}$ 0/1 matrix $A_M$ over the boolean field $GF_2$.  Each edge $e=\set{i_1,i_2,\ldots,i_k}$ gives rise to a column $\bc_e$ that has a 1 in rows $i_1,i_2,\ldots,i_k$ and a 0 everywhere else. We view the set of $m$ edges acquired by Maker as an $n\times m$ 0/1 sub-matrix $A_m$. Maker's aim to build a matrix $A_m$ of ``full rank'', i.e. $n$ if $k$ is odd and $n-1$ if $k$ is even (note that when $k$ is even, rank $n$ is not possible since summing all rows gives the zero vector). We prove
\begin{theorem}\label{thmatrix}
Let $b=\frac{\b N}{\log n}$. Then for any fixed $\e > 0$, Maker wins the Matrix Rank Game if $\b\leq 1-\e$ and Breaker wins if $\b\geq 1+\e$.
\end{theorem}
Next we consider the {\em Berge Hamiltonicity Game}. A Berge Hamilton cycle (BHC) is a sequence $(v_1,e_1,v_2,e_2,\ldots,v_{n-1},e_n,v_n)$ of distinct vertices and edges where $V(H)=\set{v_1,v_2,\ldots,v_n}$ and $\set{v_i,v_{i+1}}\subseteq e_i$ for $i=1,2,\ldots,n$ (indices taken modulo $n$). In this  Game Maker wants to construct a BHC. Krivelevich \cite{HamKriv} asymptotically estimated the threshold bias in the graph case (note that for graphs, a BHC is just a Hamilton cycle). We adapt his proof to hypergraphs. 
\begin{theorem}\label{thBHC}
Let $b=\frac{\b N}{\log n}$. Then for any fixed $\e > 0$, Maker wins the BHC Game if $\b\leq 1-\e$ and Breaker wins if $\b\geq 1+\e$.
\end{theorem}
Our next theorem concerns the Hamilton $\ell$-cycle Game in $H_{n,k}$. A {\em Hamilton $\ell$-cycle} is, for some ordering $v_1, \ldots, v_n$ of the vertices, a sequence of edges $\set{v_1, \ldots, v_k}, \set{v_{k-\ell+1}, \ldots, v_{2k-\ell}} \ldots,$ subscripts taken modulo $n$, so that each edge in the sequence intersects the previous one in $\ell$ vertices. A {\em perfect matching} is a set of edges $E_1,E_2,\ldots,E_{n/k}$ whose union is $[n]$ and for which $E_i\cap E_{j}=\es$ for all $i\neq j$. For convenience, since we will treat them together, we view a perfect matching as a Hamilton $0$-cycle. 
\begin{theorem}\label{MBPM}
Fix $\ell, k$ with $\ell < k/2$. Let $b_0=\frac{n}{2k\log n}$. Then, Maker wins the Hamilton $\ell$-cycle Game if $b\leq b_0$.
\end{theorem}
The lower bound on the threshold bias is rather weak. For an upper bound we can take $b\sim N/\log n$ as in Theorems \ref{thdegree}, \ref{thmatrix} and \ref{thBHC}.

A small tweak to the proof of Theorem \ref{MBPM} yields a more colorful version. In the Rainbow  Hamilton $\ell$-cycle Game Maker gives her acquired edge a color from a set $Q$ of size $D=n/(k-\ell)$. Her aim is to build a rainbow Hamilton $\ell$-cycle i.e. one in which edge has a different color.
\begin{theorem}\label{RMBPM}
Fix $\ell, k$ with $\ell < k/2$ and let $Q$ be a $D$-set of colors. Then, Maker wins the Rainbow Hamilton $\ell$-cycle Game if $b\leq b_0$.
\end{theorem}
In Section \ref{sec:degree} we prove Theorem \ref{thdegree}. In Section \ref{sec:rank} we prove Theorem \ref{thmatrix}. In Section \ref{sec:BHC} we prove Theorem \ref{thBHC}. In Section \ref{sec:MBPM} we prove Theorems \ref{MBPM} and \ref{RMBPM}.

We follow a strategy as laid out in Krivelevich \cite{HamKriv} and Ferber, Krivelevich and Naves \cite{FKN}. Here Maker builds a random graph that w.h.p. contains the desired structure. Our contribution is to extend this idea to hypergraphs and prove Theorems \ref{thdegree} to \ref{RMBPM}. 
\section{Minimum Degree Game}\label{sec:degree}
\subsection{Breaker strategy}
Suppose that $b\geq \frac{(1+\e)N}{\log n}$. We will show that Breaker can take every edge incident with some particular vertex, which means that Breaker actually wins all of the games listed in Theorems \ref{thdegree} -- \ref{thBHC}. Breaker's strategy here will have two phases. 

Breaker's goal in Phase 1 is to build a set $S$ of $s:=n / \log n$ vertices with the following properties: by the end, none of the vertices in $S$ are incident with any of Maker's edges, and Breaker has taken every edge which is incident with two vertices of $S$. We follow a strategy based on that of Chv\'atal and Erd\H{o}s \cite{CE} for graphs.

To analyze Phase 1, suppose Breaker has some current set $S$ with the required properties (in the beginning $S$ is empty). We will show that Breaker can increase the size of $S$ by one each turn. Assuming that Breaker has succeeded so far in growing $S$ by one vertex at each turn, but $|S|<s$, there have only been $|S|=o(n)$ turns so far and so Maker has only touched $o(n)$ vertices. Breaker chooses two vertices $x, y$ that are untouched by Maker. Breaker takes every edge containing $x$ and $y$, and every edge containing one of $x, y$ and another vertex from $S$. Breaker can do all of this in one turn since 
$b > (1+2|S|)\binom {n}{k-2}$. Breaker then adds $x, y$ to $S$. On Maker's turn, Maker can touch at most one vertex from $S$, and so playing by this strategy Breaker can keep growing $S$ by one vertex each turn. Thus, Phase 1 succeeds. 

In Phase 2, Breaker's goal is to take every remaining edge from some vertex in $S$. Phase 2 will be an instance of the Box Game of Chv\'atal and Erd\H{o}s \cite{CE}. In the Box Game, we have $x$ boxes, each containing $y$ balls. Two players, BoxMaker and BoxBreaker, take turns where BoxMaker takes $b$ balls per turn and BoxBreaker only takes one ball. BoxMaker wins if they take every ball in some box. Theorem 2.1 in \cite{CE} implies that BoxMaker wins if $b \ge y / \log x$. For our present Minimum Degree Game, Breaker will play the role of BoxMaker. For each vertex $v\in S$ we have a ``box'' consisting of all the edges containing $v$, so each box has at most $N$ elements. Since $b \ge N/ \log s \sim N / \log n$, Breaker wins. 

\subsection{Maker strategy}
We will prove something stronger than what is strictly necessary for Therem \ref{thdegree}. 
\begin{theorem}\label{thdegreegeneral}
Let $N=\binom{n-1}{k-1}$ and breaker's bias be $b$. Suppose that $1\leq b\leq \frac{(1-\e')\a N}{km+\log n}$ for $\a, \e' \in (0, 1)$. There exists a strategy for Maker to obtain minimum degree $m$ in at most $mn$ rounds. Furthermore, Maker's strategy involves some arbitrary choices. In particular, in any round, Maker chooses a vertex $v$ (specified by the strategy) and then arbitrarily chooses an edge incident containing $v$ from a set of size at least $(1-\a)N$.
\end{theorem}
The arbitrary choices allowed in Maker's strategy will be used to prove Theorems \ref{thmatrix}, \ref{thBHC}. In those proofs, Maker will make each arbitrary choice uniformly at random. To prove Theorem \ref{MBPM}, we will prove a version where the board is an arbitrary $D$-regular graph ($k=2$).
\begin{theorem}\label{thdegreegeneral1}
Let $G$ be an arbitrary $n$-vertex, $D$-regular graph, i.e $k=2$ and let breaker's bias be $b$. Suppose that $b\leq \frac{(1-\e')\a D}{2m+\log n}$ for $\a, \e' \in (0, 1)$. There exists a strategy for Maker to obtain minimum degree $m$ in at most $mn$ rounds. Furthermore, Maker's strategy involves some arbitrary choices. In particular, in any round, Maker chooses a vertex $v$ (specified by the strategy) and then arbitrarily chooses an edge incident with  $v$ from a set of size at least $(1-\a)D$.
\end{theorem}
The fact that Maker wins in Theorem \ref{thdegree} when $m=o(\log n),\b \le 1-\e$  follows from Theorem \ref{thdegreegeneral}, with $\a=(1-\e'/2), b=(1-\e')N/\log n$, where $\e'$ is chosen sufficiently small depending on $\e$, and $m=O(1)$.

\begin{proof}[Proof of Theorem \ref{thdegreegeneral}]
We adapt the argument of Gebauer and Szabo \cite{GS} from the graph case.
 We let $H_M,H_B$ denote the subgraphs of $H$ with the edges taken by Maker, Breaker respectively. Let $d_B(v), d_M(v)$ denote the degree of $v$ in $H_B, H_M$ respectively. At each turn, Maker will choose some vertex $v$ and take an edge $e$ incident with $v$, which we will call {\em easing} $v$. Of course this edge $e$ also contains other vertices, but we only say that one of them is eased per turn. Let $d_M^+(v)$ denote the number of times a vertex $v$ has been eased so far (so $d_M(v) \ge d_M^+(v)$). Let $\dg(v):=d_B(v)-kbd^+_M(v)$ be the {\em danger} of vertex $v$ at any time. A vertex is {\em dangerous} for Maker if $d^+_M(v)<m$. 

{\bf Maker's Strategy:} In round $i$, choose a dangerous vertex $v_i$ of maximum danger and choose an edge randomly from those edges incident with $v_i$ that are not already taken. 

\begin{claim}\label{clm:smalldb}
 Maker can ensure that $d_B(v)\le \a N$ for all dangerous $v\in V$.
\end{claim} 
\begin{proof}[Proof of Claim \ref{clm:smalldb}]
 Let $M_i,B_i$ denote Maker and Breaker's $i$th moves. Suppose for contradiction that after $B_{g-1}$ there is a dangerous vertex $v_g$ such that $d_B(v_g)>\a N$. Let $J(i):=\set{v_{i+1},\ldots,v_g}$. Next define 
\[
\bdg(M_i)=\frac{\sum_{v\in J(i-1)}\dg(v)}{|J(i-1)|}\quad\text{ and }\quad\bdg(B_i)=\frac{\sum_{v\in J(i)}\dg(v)}{|J(i)|},
\]
computed before the $i$th moves of Maker, Breaker respectively. 

Then $\bdg(M_1)=0$ and $\bdg(M_g)=\dg(v_g)> \a N-kbm$. Let $a_s(i)$ be the number of edges $e$ claimed by Breaker in his first $i$ moves such that  $|J(i)\cap e|=s$. We have
\begin{lemma}\label{lem1a}
\begin{align}
\bdg(M_i)&\geq \bdg(B_i).\label{appa1}\\
\bdg(M_i)&\geq \bdg(B_i)+\frac{kb}{|J(i)|},\text{ if }J(i)=J(i-1).\label{appa2}\\
\bdg(B_i)&\geq \bdg(M_{i+1})-\frac{kb}{|J(i)|}\label{appa3}\\
\bdg(B_i)&\geq \bdg(M_{i+1})-\frac{b+\s(i)}{|J(i)|}.\label{appa4}
\end{align}
where 
\beq{eqn:sigmadef}
{
\s(i):=\sum_{s=2}^k(s-1)\brac{a_s(i)-a_s(i-1)+\binom{|J(i)|}{s-1}\binom{n}{k-s}}.
}
\end{lemma}
\begin{proof}[Proof of Lemma \ref{lem1a}]
Equation \eqref{appa1} follows from the fact that a move by Maker does not increase danger. Equation \eqref{appa2} follows from the fact that if $v_i\in J_{i-1}$ then its danger, which is a maximum, drops by $kb$. Equation \eqref{appa3} follows from the fact that Breaker takes at most $b$ edges contained in $J_i$. For equation \ref{appa4} note that  $\bdg(M_{i+1})-\bdg(B_i)$ is $|J(i)|^{-1}$ times the increase in $\sum_{v\in J(i)}\dg(v)$ due to Breaker's additional choices. This will equal $|J(i)|^{-1}$ times the increase in $d_B(J(i))$. If Breaker chooses an edge $e$ such that $|e\cap J(i)|=s$ then this adds $s$ to $d_B(J(i))$. Let $b_s(i)$ be the number edges of this form, so that the change in $d_B(J(i))$ is $\sum_{s=1}^ksb_s(i)$. Then 
\beq{asi}{
a_s(i)-a_s(i-1)\leq b_s(i)-A_s,
}
 where $A_s=\binom{|J(i)|}{s-1}\binom{n}{k-s}$. The term $A_s$ arises as a bound on the number of edges lost due to the deletion of the vertex $v_{i-1}$. So,
 \[
 \sum_{s=1}^ksb_s(i)=\sum_{s=1}^kb_s(i)+\sum_{s=2}^k(s-1)b_s(i)\leq b+\s_i,\text{ from \eqref{asi}.}
 \]
 This completes the proof of Lemma \ref{lem1a}.
\end{proof}
It follows that 
\begin{align}
\bdg(M_i)&\geq \bdg(M_{i+1})\text{ if }J(i)=J(i-1).\label{appa5}\\
\bdg(M_i)&\geq \bdg(M_{i+1})-\min\set{\frac{kb}{|J(i)|},\frac{b+\s(i)}{|J(i)|}}.\label{appa6}
\end{align}
Next let $1\leq i_1\leq \cdots\leq i_r\leq g-1$ be the indices where $J(i)\neq J(i-1)$. Then we have $|J(i_k)|=r-k+1$ for each $k$, and $|J(i_1-1)|=|J(0)|=r+1$. Let $\ell=\rdup{\frac{n}{\log n}}$ and assume first that $r\geq \ell$ and then use the first minimand in \eqref{appa6} for $i_1,\ldots,i_{r-\ell}$ and the second minimand otherwise.
\begin{align}
    0=\bdg(M_1)&\geq \bdg(M_g)-\frac{b+\s(i_r)}{|J(i_r)|}-\cdots- \frac{b+\s(i_{r-\ell+1})}{|J(i_{r-\ell+1})|} 
-\frac{kb}{|J(i_{r-\ell})|}-\cdots-\frac{kb}{|J(i_1)|}\nonumber\\
&> \a N-kbm-\frac{b+\s(i_r)}{1}-\cdots- \frac{b+\s(i_{r-\ell+1})}{\ell} 
-\frac{kb}{\ell+1}-\cdots-\frac{kb}{r}\nonumber\\
&= \a N-kbm - b \log \ell  - kb \log \frac r\ell - \frac{\s(i_r)}{1}-\cdots- \frac{\s(i_{r-\ell+1})}{\ell} +O(b). \label{eqn:dangerbound}
\end{align}
Now using \eqref{eqn:sigmadef} we have that 
\begin{align}
 \sum_{j=r-\ell+1}^r \frac{\s(i_{j})}{r-j+1} &= \sum_{j=r-\ell+1}^r \frac{\sum_{s=2}^k(s-1)\brac{a_s(i_{j})-a_s(i_{j}-1)+\binom{r-j+1}{s-1}\binom{n}{k-s}}}{r-j+1}  \nonumber \\
 & = \sum_{s=2}^k (s-1) \sum_{j=r-\ell+1}^r  \frac{a_s(i_{j})-a_s(i_{j}-1)}{r-j+1} + \sum_{j=r-\ell+1}^r \sum_{s=2}^k \frac{\binom{r-j+1}{s-1}\binom{n}{k-s}}{r-j+1}\label{eqn:sigmabound}
\end{align}

Regarding the first sum in \eqref{eqn:sigmabound}, we claim that 
\beq{sumsum}{
\sum_{j=r-\ell+1}^r\frac{a_s(i_j)-a_s(i_j-1)}{r-j+1}=\sum_{j=1}^\ell\frac{a_s(i_{r+1-j})-a_s(i_{r+1-j}-1)}{j}\leq 0.
}
Indeed, observe first that $a_s(i_{r+1-j}-1)\geq a_s(i_{r-j}),j\geq0$ follows from $J(i_{r-j})=J(i_{r+1-j}-1)$. Also, $a_s(i_r)=0$ because $J(i_r)=J(g-1)=\set{v_g}$. And then \eqref{sumsum} becomes 
\[
a_s(i_r)+\sum_{j=1}^{\ell-1}\brac{\frac{a_s(i_{r-j})}{j+1}-\frac{a_s(i_{r+1-j}-1)}{j}}-\frac{a_s(i_{r+1-\ell}-1)}{\ell}\leq 0.
\]

Regarding the second sum in \eqref{eqn:sigmabound}, note that for $s\ge 3$ and $r-\ell+1 \le j \le r$ we have $\frac{\binom{r-j+1}{s-1}\binom{n}{k-s}}{r-j+1}=O\bfrac{1}{\log n}\frac{\binom{r-j+1}{s-2}\binom{n}{k-s+1}}{r-j+1}$ and so the second sum is at most $\brac{1 + O\brac{\frac {1}{\log n}}} \ell \binom{n}{k-2}$.

Continuing from \eqref{eqn:dangerbound} we have
\begin{align*}
0 &> \a N-kbm - b \log \ell  - kb \log \frac r\ell - \brac{1 + O\brac{\frac {1}{\log n}}} \ell \binom{n}{k-2} + O(b)\\
&\geq \a N-kbm-b\log n+b(k+1)\log\log n-O\bfrac{N}{\log n}+O(b)>0
\end{align*}
since $b\leq \frac{(1-\e')\a N}{km+\log n}$. This completes the case where $r \ge \ell$.

If $r<\ell$ then we use the first minimand in \eqref{appa6} and get a similar contradiction from
\[
0\geq \a N-kbm-b\log \ell.
\]
This completes the proof of Claim \ref{clm:smalldb}.
\end{proof}
This completes the proof of Theorem \ref{thdegreegeneral}.
\end{proof}

Note that in the proof above, we actually showed that every vertex is eased $m$ times, which is stronger than Maker having minimum degree $m$. Thus we have the following.
\begin{corollary}\label{cor:eased}
    Under the assumptions of Theorem \ref{thdegreegeneral}, Maker can ensure that each vertex is eased $m$ times. 
\end{corollary}
	 The purpose of this corollary is for use in the Matrix Rank Game and in the BHC Game. In those games, the first phase of Maker's strategy will be to build a hypergraph that roughly resembles ``$m$-out'' which is done by easing each vertex $m=O(1)$ times. 
	 
Theorem \ref{thdegreegeneral1} is proved basically by replacing $N$ by $D$ in the proof of Theorem \ref{thdegreegeneral}. One can also see a complete proof in the appendix of \cite{MBdir}. (This proof replaces $m$ by $K$, otherwise the proof is basically the same as the proof of Theorem  \ref{thdegreegeneral}.
\section{Rank Game}\label{sec:rank}
Observe first that if $b\geq (1+\e)N / \log n$ then Breaker can isolate a vertex and win the game. So from now on asume that $b\leq (1-\e)N / \log n$.

 {\bf Maker's Strategy:} There will be two phases, the first of which will be to play the Minimum Degree Game. We apply Theorem \ref{thdegreegeneral} with $\e'$ small with respect to $\e$, $\a=1-\e'$ and $m=O(1)$ where $m$ is large. The arbitrary choices in the strategy will be made uniformly at random. In so doing Maker builds a random matrix $A$ with the following properties: for each $i\in[n]$ there are $m$ columns $Col_i=\set{\bc_{i,j},j=1,2,\ldots,m}$ where (i) each $\bc\in Col_i$ has a 1 in row $i$ and (ii) the other $k-1$ 1's are chosen randomly from a subset of $\binom{[n]\sm \set{i}}{k-1}$ of size at least $\a N$.
 
 A {\em dependency} will be a set of rows $S$ such that $\sum_{\br\in S}\br=0$. We have to show that w.h.p. there are no dependencies, except for $S=[n]$ in the case where $k$ is even. Let $\cD_S$ for $S\subseteq [n]$ be the event that $S$ defines a dependency. We now show that at the end of Phase 1, there are no dependencies $S$ of size $s \le \eta n$ where $\eta=\frac{1-\a}{3k}$. We have
\begin{align}
 \Pr(\cD_S)\leq \prod_{i\in S}\brac{\frac{s\binom{n-1}{k-2}}{(1-\a)\binom{n-1}{k-1}}}^m \le \brac{\frac{ks}{(1-\a)n}}^{ms}.\label{DS}
 \end{align}
 {\bf Explanation:} the second inequality is easy, so we explain the first. If $i\in S$ then each edge chosen by vertex $i$ must have an even number of vertices in $S$ (including $i$), so there must be another vertex in $S$. The number of possibilities for each such edge is at most $s\binom{n-1}{k-2}$. The denominator accounts for the edge being chosen from a set of size at least $(1-\a)\binom{n-1}{k-1}$. 
 
 So, from \eqref{DS},
 \begin{align*}
 \Pr(\exists \text{ a dependency }S,|S|\leq \eta n)&\leq\sum_{s=1}^{\eta n}\binom{n}{s}\bfrac{ks}{(1-\a)n}^{ms}\\
 &\leq \sum_{s=1}^{\eta n}\brac{\bfrac{s}{n}^{m-1}\cdot\frac{k^{m}e}{(1-\a)^m}}^s=o(1),
\end{align*}
 since $\eta=\frac{1-\a}{3k}$ and $m$ is sufficiently large. 

In Phase 2, Maker just chooses uniformly at random from the available set of edges. We will argue that in Phase 2, Maker can eliminate all remaining dependencies $S$. 
 
First assume that $\eta n \le |S| \le (1-\eta)n$. There are $W_s=s\binom{n-s}{k-1}=\Omega(n^k)$ edges with exactly one vertex in $S$, and we call this set of edges $E_S$. At the start of Phase 2, at most $m(b+1)N=o(n^k)$ have been acquired by Maker or Breaker. In the next $n \log \log n$ rounds, at most $(b+1) n \log \log n = o(n^k)$ additional edges will be acquired by the players. So the probability that Maker avoids $E_S$ during these rounds is at most 
\[
\brac{1-\frac{W_s - o(n^k)}{\binom{n}{k}}}^{n \log \log n} \le \exp \brac{-\Om(n \log \log n)}.
\]
Taking a union bound over all choices of $S$, we see that w.h.p.~there will be no dependencies $S$ of size $|S| \le (1-\eta)n$ at the end of Phase 2. If $k$ is even then we are done since the complement of a dependency is a dependency.

For odd $k$, and a dependency with $|S|\ge (1-\eta)n$ we focus on $T=[n]\sm S$ during Phase 1. This will be very similar to our treatment of $\cD_S$ on line \eqref{DS}. Note that $\sum_{\br\in T}={\bf 1}$, where ${\bf 1}=[1,1,\ldots,1]$. Let now $\cO_T $ be the event that $\sum_{\br\in T}={\bf 1}$. Then, where $t=|T|$,
\beq{OT}{
\Pr(\cO_T)\leq\prod_{i\notin T}\brac{\frac{t\binom{n-1}{k-2}}{(1-\a)\binom{n-1}{k-1}}}^m.
 }
 {\bf Explanation:} if $i\notin T$ then each edge/column chosen by vertex $i$ must contain a vertex in $T$.
 
 The rest of the argument follows that of the above for the events $\cD_S$. This completes the proof of Theorem \ref{thmatrix}.

\section{BHC Game}\label{sec:BHC}

\begin{definition}
A hypergraph is an \emph{$(r, \alpha)$-expander} if for all disjoint sets of vertices $X$ and $Y$, if $|Y| < \alpha |X|$ and $|X| \leq r$, then there is an edge, $e$, such that $|e \cap X| = 1$ and $e \cap Y = \emptyset$.
\end{definition}

\begin{definition}\label{def:booster}
For a hypergraph $\cH$, a \emph{booster} is an edge such that either $\cH \cup e$ has a longer (Berge) path than $\cH$ or $\cH \cup e$ is (Berge) Hamiltonian.
\end{definition}

\begin{lemma}[Lemma 2.1 in \cite{BBDS}] \label{lem:boosters}
There exists a constant $c_k > 0$ such that if $\cH$ is a connected $(r,2)$-expander $k$-graph on at least $k+1$ vertices, then $\cH$ is Hamiltonian, or it has at least $r^2 n^{k-2} c_k$ boosters.
\end{lemma}

\begin{lemma}\label{lem:expander}
There exist constants $\e', m>0$ ($\e'$ will be small, and $m$ will be large) satisfying the following. Let $\cH$ be the hypergraph obtained by Maker playing according to Theorem \ref{thdegreegeneral} using the parameter $\a=1-\e'$.  Then w.h.p. $\cH$ is an $(\e n/10k, 2)$-expander.
\end{lemma}
\begin{proof}[Proof of Lemma \ref{lem:expander}]
$m$ is a constant and so Maker will win the Minimum Degree Game before any significant proportion of edges are taken by the players. Fix some $X, Y$ with $|X|=x\le \e n/10k$ and $|Y|=2x$. On each step when we ease a vertex in $X$, the probability of choosing an edge that contains an additional vertex (other than the one we are easing) in $X \cup Y$ is at most $\sim 3x \binom{n}{k-2} / (1-\a)N < 3k x / 2\e n$. 

By the union bound, the probability that there exist some $X, Y$ such that $\cH$ does not have any edge $e$ with exactly one vertex in $X$ and no vertices in $Y$ is at most 
\begin{align*}
    \sum_{x = 1}^{\e n/10k} \binom{n}{x} \binom{n}{2x} \brac{\frac{3kx}{2\e n}}^{mx} & \le \sum_{x = 1}^{\e n/10k} \brac{\frac{ne}{x}}^{3x} \brac{\frac{3kx}{2\e n}}^{mx} =o(1)
\end{align*}
\end{proof}

Maker's strategy is as follows. First play the Minimum Degree Game from Lemma \ref{lem:expander} to obtain $\cH$. Since $\cH$ is a $(\e n/10k, 2)$-expander, each connected component has at least $\e n/10k$ vertices, so there are at most $10k / \e$ components. After winning the Minimum Degree Game, Maker's next goal is to choose edges to connect $\cH$, which takes at most $10k /\e$ turns. Now $\cH$ is a connected $(\e n/10k, 2)$-expander, and Lemma \ref{lem:boosters} tells us there are $\Om(n^{k})$ boosters. We are still so early in the game that only $o(n^k)$ edges have been taken by the players, so almost all the boosters are available. Now Maker chooses a booster every turn for at most $n$ turns, winning the game. This proves Theorem \ref{thBHC}.
 \section{Perfect Matchings and Hamilton $\ell$-cycles}\label{sec:MBPM}
 {\bf Maker's Strategy for Theorem \ref{MBPM}:} Maker starts by severely restricting the board. Let $[n]$ be the vertex set of the board. Let $A_i=\set{(k-\ell - 1)(i-1)+1,(k-\ell -1)(i-1)+2,\ldots,(k-\ell-1)i+\ell}$ for $i=1,2,\ldots,D=\frac{n}{k-\ell}$. The elements of $A_i$ are taken modulo $n-D$, and so the $A_i$ form a $(k-1)$-uniform Hamilton $\ell$-cycle on the vertex set $\set{1, \ldots, n-D}$. (When $\ell=0$ this is a perfect matching.) Let $B=\{n-D+1,$ $n-D+2,\ldots,n\}$ be the vertices not in any $A_i$. Maker will restrict the board by only taking edges of the form $A_i \cup \{x\}$ where $x \in B$. Breaker is of course free to take other edges but this will not matter. 
 
 Now let $\G$ be the complete bipartite graph with vertex sets $A=\set{A_1,A_2,\ldots,A_D}$. To be clear, $|A|=|B|=D$. Each edge $(A_i,j)$ of $\G$ corresponds to the edge $A_i\cup\{j\}$ of $H_{n,k}$. A perfect matching $\set{(A_i,\f(i)),i=1,2,\ldots,D}$ corresponds to Hamilton $\ell$-cycle $\set{E_i\cup\set{\f(i)},i=1,2,\ldots,D}$ in $H$.
 
 Maker will make her choices as if she were playing the minimum degree game on $\G$ with $m=10,\a=1/11$. If she wants to ease $E_i\in A$ then she chooses $v\in E_i$ and a random $w\in B$ such that the edge $X=E_i\cup \set{w}$ is currently available and adds $X$ to $E(H_M)$. Similarly, if she wants to ease $w\in B$ then she chooses a random $i$ such that the edge $X=E_i\cup \set{w}$ is currently available and adds $X$ to $E(H_M)$. We claim that w.h.p. Maker claims a perfect matching in $\G$.
 
 We apply Hall's theorem. We show first that if $\cB_1=\set{\exists S\subseteq A,|S|\leq D/2,|N_\G(S)|<|S|}$ then 
 \beq{Hall1}{
 \Pr(\cB_1)\leq \sum_{s=5}^{D/2}\binom{D}{s}^2\bfrac{11s}{10D}^{10s}\leq \sum_{s=5}^{D/2}\brac{\frac{D^2e^2}{s^2}\cdot\bfrac{ 11s}{10D}^{10}}^s=o(1).
 }
 {\bf Explanation:} fix sets $S\subseteq A,T\subseteq B$ of size $s$. Then the probability that Maker chooses $w\in T$ when easing $v\in E\in S$ is at most $\frac{s}{(1-\a)D}=\frac{11s}{10D}$.
 
 If $|S|>D/2$ and $|N_\G(S)|<|S|$ then $|N_\G(B\setminus N_\G(S))|<|B\setminus N_\G(S)|$ and we proceed as in \eqref{Hall1}. This completes the proof of Theorem \ref{MBPM}. 
 
 Now consider Theorem \ref{RMBPM}. We have a set of colors $Q=\{c_1, \ldots, c_D\}$, and whenever Maker takes an edge she will assign it the color $c_i$ where $n-D+i$ is the vertex in $B$ which Maker's edge contains. Thus our Hamilton $\ell$-cycle will be rainbow and the rest of the proof follows that of the uncolored version.
\section{Open questions}
While some of our results are tight, there are several open questions that immediately spring to mind. Theorem \ref{thmatrix} could probably be extended to fields other than $GF_2$. The bounds on the bias $b$ seem quite weak in Theorems \ref{MBPM}, \ref{RMBPM} and they do not include the cases where $k/2\leq \ell<k$.

\end{document}